\documentclass{amsart}
\usepackage{amsmath,amsthm,amsfonts,amssymb,mathrsfs}
\date{\today}
\usepackage{color}
\usepackage{hyperref}
\newtheorem{theorem}{Theorem}
\newtheorem{proposition}[theorem]{Proposition}
\newtheorem{corollary}[theorem]{Corollary}

\theoremstyle{definition}

\newtheorem{example}[theorem]{Example}
\newtheorem{remark}[theorem]{Remark}
\begin{document}

\title[Maximal subgroups in countably compact
topological semigroups] {The continuity of the inversion and the
structure\\ of maximal subgroups in countably compact
\\ topological semigroups}

\author[O. Gutik]{Oleg~Gutik}
\address{Department of Mechanics and Mathematics, Ivan Franko Lviv
National University, Universytetska 1, Lviv, 79000, Ukraine}
\email{o\_\,gutik@franko.lviv.ua, ovgutik@yahoo.com}
\author[D. Pagon]{Du\v{s}an~Pagon}
\address{Institute of Mathematics, Physics and Mechanics, and
Faculty of Natural Sciences and Mathematics, University of
Maribor, Gosposvetska 84, Maribor 2000, Slovenia}
\email{dusan.pagon@uni-mb.si}
\author[D. Repov\v s]{Du\v{s}an~Repov\v{s}}
\address{Faculty of Mathematics and Physics, and
Faculty of Education, University of Ljubljana,
P.~O.~B. 2964, Ljubljana, 1001, Slovenia}
\email{dusan.repovs@guest.arnes.si}



\begin{abstract}
In this paper we search for conditions on a countably compact
(pseudo-compact) topological semigroup under which: 
(i) each maximal
subgroup $H(e)$ in $S$ is a (closed) topological subgroup in $S$;
(ii) the Clifford part $H(S)$(i.e. the union of all maximal subgroups) of
the semigroup $S$ is a closed subset in $S$;  
(iii)
the inversion
$\operatorname{inv}\colon H(S)\rightarrow  H(S)$ is continuous; 
and
(iv) the projection $\pi\colon H(S)\rightarrow E(S)$, $\pi\colon
x\longmapsto xx^{-1}$, onto the subset of idempotents $E(S)$ of
$S$, is continuous.
\end{abstract}

\maketitle



In this paper all topological spaces will be assumed to be
Hausdorff. We shall follow the terminology
of~\cite{CarruthHildebrantKoch1983-1986, CliffordPreston1961-1967,
Engelking1989}.  We shall
denote the cardinality of
continuum
by $\mathfrak{c}$ and
the topological closure of
subset $A$ in a topological space by $\overline{A}$
We shall call a $T_3$-space a regular
topological space.

A topological space $X$ is said to be
\emph{countably compact} if any
countable open cover of $X$ contains a finite
subcover~\cite{Engelking1989}. 
A topological space $X$ is called
\emph{pseudocompact} if each continuous real-valued function on
$X$ is bounded~\cite{Engelking1989}. A topological space $X$ is
said to be
\emph{sequential} if each non-closed subset $A$ of $X$
contains a sequence of points $\{ x_n\}_{n=1}^\infty$ that
converges to some point $x\in X\setminus A$. 
Obviously, a
topological space $X$ is sequential if a subset $A$ of $X$ is
closed if and only if together with any convergent sequence $A$
contains its limit~\cite{Engelking1989}. 
A topological space $X$
is called \emph{sequentially compact} if each sequence $\{
x_n\}_{n=1}^\infty\subset X$ has a convergent
subsequence~\cite{Engelking1989}.

We recall that the Stone-\v{C}ech compactification of a Tychonoff
space $X$ is a compact Hausdorff space $\beta{X}$ containing $X$
is a dense subspace so that each continuous map $f\colon
X\rightarrow Y$ to a compact Hausdorff space $Y$ extends to a
continuous map $\overline{f}\colon \beta{X}\rightarrow
Y$~\cite{Engelking1989}.

A \emph{semigroup} is a set with a binary associative operation.
An element $e$ of a semigroup $S$ is called an \emph{idempotent}
if $e \circ e=e$. If $S$ is a~semigroup, then  we denote the
subset of all idempotents of~$S$
by $E(S)$.
A semigroup $S$ is called {\em inverse} if for any $x\in S$ there
exists a unique $y\in S$ such that $xyx=x$ and $yxy=y$. Such an
element $y$ is called {\em inverse} of $x$ and is denoted by
$x^{-1}$. If $S$ is an inverse semigroup, then the map which takes
$x\in S$ to the inverse element of $x$ is called the {\em
inversion} and will be denoted by $\operatorname{inv}$.

If $S$ is a semigroup and $e$ is an idempotent in $S$, then $e$
lies in the maximal subgroup $H(e)$ with the identity $e$. If a
semigroup $S$ is a union of groups then $S$ called
\emph{Clifford}. On a Clifford semigroup
$S=\bigcup\{H(e)\mid e\in
E(S)\}$ 
the inversion $\operatorname{inv}\colon
S\rightarrow S$ is defined
which maps each element $x\in H(e)$ to its inverse
element $x^{-1}$ in $H(e)$. 
We also
observe that on any Clifford
semigroup the \emph{projection} $\pi\colon S\rightarrow E(S)$,
$\pi(x)=x\cdot x^{-1}$, is defined. For a semigroup $S$ let
\begin{equation*}
\begin{split}
    H(S)= & \bigcup_{e\in E(S)}\{H(e)\mid H(e) \mbox{~is a maximal
    subgroup in~} S \mbox{~with identity~} e\}=\\
        =& \{ s\in S\mid \mbox{~there exists~}\; y\in S \;\mbox{~such
        that~}\; xy=yx,\; xyx=x,\; yxy=y\}.
\end{split}
\end{equation*}

A topological space $S$ which is algebraically a semigroup with a
continuous semigroup operation is called a {\em topological
semigroup}. A {\em topological inverse semigroup} is a topological
semigroup $S$ that is algebraically an inverse semigroup with
continuous inversion. If $\tau$ is a topology on a (inverse)
semigroup $S$ such that $(S,\tau)$ is a topological (inverse)
semigroup, then $\tau$ is called a ({\em inverse}) {\em semigroup
topology} on $S$.
By a \emph{paratopological group} we understand a pair $(G,\tau)$
consisting of a group $G$ and a topology $\tau$ on $G$ making the
group operation on $G$ continuous. A paratopological group $G$
with continuous inversion is
called a \emph{topological group}.

Finite semigroups and compact topological semigroups have similar
properties. For example every finite semigroup and every compact
topological semigroup contain idempotents and minimal
ideals~\cite{Wallace1955}, which are completely simple
semigroups~\cite{Suschkewitsch1928, Wallace1956}, and every
($0$--)simple compact topological (and hence finite) semigroup is
completely ($0$--)simple \cite{Paalman-de-Miranda1964,
Suschkewitsch1928, Wallace1956}. Also, a cancellative compact
topological (and hence finite) semigroup is a topological
group~\cite{Numakura1952}.

Compact topological semigroups do not contain the bicyclic
semigroup~\cite{HildebrantKoch1988}.  
Gutik and Repov\v{s} 
\cite{GutikRepovs2007}
proved that a countably compact topological
inverse semigroup does not contain the bicyclic semigroup. Banakh,
Dimitrova and Gutik 
\cite{BanakhDimitrovaGutik2009x} showed
that no
topological semigroup $S$ with countably compact square
$S\times S$ contains
the bicyclic semigroup. 
They also constructed in
\cite{BanakhDimitrovaGutik2009xx}
a consistent
example a countably compact topological semigroup $S$ which
contains the bicyclic semigroup.

It is well known that the closure of a (commutative) subsemigroup
of a topological semigroup is a (commutative) subsemigroup
\cite[Vol.~1, P.~9]{CarruthHildebrantKoch1983-1986}. 
Note that the
closure of a subgroup in a topological semigroup is not
necesarily a subgroup. But in the case when $S$ is a compact
topological semigroup or a topological inverse semigroup, the
closure of a subgroup in $S$ is a subgroup, moreover every maximal
subgroup of $S$ is closed (see \cite[Vol.~1,
Theorem~1.11]{CarruthHildebrantKoch1983-1986} and
\cite{EberhartSelden1969}). 

Also, every compact subgroup of a
topological semigroup with induced topology is a topological
semigroup. The results when the inversion is continuous in a
topological semigroup which is algebraically a group (i.~e. a
paratopological group) have been extended to some classes of
``compact-like'' paratopological groups, in particular: regular
locally compact paratopological groups \cite{Ellis1953}, regular
countably compact paratopological groups
\cite{RavskyReznichenko2002}, quasi-regular pseudo-compact
paratopological groups \cite{RavskyReznichenko2002}, topologically
periodic Hausdorff countably compact paratopological groups
\cite{BokaloGuran1996}, \v Cech-complete paratopological groups
\cite{Brand1982}, strongly Baire semi-topological groups
\cite{KenderovKortezovMoors2001}. 

On the other hand, Ravsky~\cite{Ravsky2003}, 
using
a result of Koszmider, Tomita and Watson
\cite{KoszmiderTomitaWatson2000}, constructed a
\textsf{MA}-example of a Hausdorff countably compact
paratopological group failing to be a topological
group. Also Grant~\cite{Grant1993} and
Yur'eva~\cite{Yuryeva1993} showed that a Hausdorff cancellative
sequential countably compact topological semigroup is a
topological group. Bokalo and Guran 
\cite{BokaloGuran1996}
established that an analogous theorem is true for cancellative
sequentially compact semigroups. Robbie and
Svetlichny~\cite{RobbieSvetlichny1996} 
constructed a
\textsf{CH}-example of a countably compact topological semigroup
which is not a topological group.

In summary (see \cite{CarruthHildebrantKoch1983-1986}), for
a compact topological semigroup $S$ the following conditions hold:
\begin{itemize}
    \item[(1)] each maximal subgroup $H(e)$ in $S$ is a compact
               topological subgroup in $S$;
    \item[(2)] the subset $H(S)$ is closed in $S$;
    \item[(3)] the inversion map $\operatorname{inv}\colon
               H(S)\rightarrow  H(S)$ is continuous; and
    \item[(4)] the projection $\pi\colon H(S)\to E(S)$ is continuous.
\end{itemize}

Since sequential compactness, countable compactness and
pseudo-compactness are generalization of compactness, it is
natural to pose the following question: \emph{For which
compact-like topological semigroups do
the conditions $(1)$--$(4)$ above
hold?}

In this paper we shall answer this question
by giving sufficient
conditions on a countably compact (pseudo-compact) topological
semigroup under which: 
(i) each maximal subgroup $H(e)$ in $S$ is a
(closed) topological subgroup in $S$; 
(ii) the Clifford part $H(S)$ of
the semigroup $S$ is a closed subset in $S$; and
(iii) the inversion
$\operatorname{inv}\colon H(S)\rightarrow  H(S)$ and the
projection $\pi\colon H(S)\to E(S)$ are continuous.


A topological group $G$ is called \emph{totally bounded} if for
any open neighbourhood $U$ of the identity $e$ of $G$ there exists
a finite subset $A$ in $G$ such that $A\cdot U=G$ (see~\cite{Weil1937}).

\begin{theorem}\label{th1}
Let $S$ be a Tychonoff topological semigroup with the
pseudocompact square $S\times S$. Then $S$ embeds into a compact
topological semigroup and the following conditions hold:
\begin{itemize}
    \item[$(i)$] the inversion $\operatorname{inv}\colon
                 H(S)\rightarrow H(S)$ is continuous;
    \item[$(ii)$] the projection $\pi\colon
                 H(S)\rightarrow E(S)$ is continuous; and
    \item[$(iii)$] for each idempotent $e\in E(S)$ the maximal
                 subgroup $H(e)$ is a totally bounded topological
                 group.
\end{itemize}
\end{theorem}

\begin{proof}
By Theorem~1.3 from~\cite{BanakhDimitrova2009xx},
for any topological
semigroup $S$ with the pseudocompact square $S\times S$ the
semigroup operation $\mu\colon S\times S\rightarrow S$ extends to
a continuous semigroup operation $\beta\mu\colon\beta S\times\beta
S\rightarrow\beta S$, so $S$ is a subsemigroup of the compact
topological semigroup $\beta S$.

$(i)$ Let
\begin{equation*}
\operatorname{Gr}_{\operatorname{inv}}(H(\beta S))=\{(x, y)\in
S\times S\mid y=x^{-1}\}
\end{equation*}
be the graph of the inversion in $H(\beta S)$. Since $\beta S$ is
a topological semigroup and
\begin{equation*}
\operatorname{Gr}_{\operatorname{inv}}(H(\beta S))=\{(x, y)\in
S\times S\mid xyx=x, yxy=y\mbox{ and } xy=yx\},
\end{equation*}
the graph $\operatorname{Gr}_{\operatorname{inv}}(H(\beta S))$ is
a compact subset of $\beta S\times\beta S$.

Consider the natural projections $\operatorname{pr}_1\colon \beta
S\times\beta S\to\beta S$ and $\operatorname{pr}_2\colon\beta
S\times\beta S\to\beta S$ onto the first and the second
coordinates, respectively. It follows from the compactness of
$\operatorname{Gr}_{\operatorname{inv}}(H(\beta S))$ that
$\operatorname{pr}_1\colon
\operatorname{Gr}_{\operatorname{inv}}(H(\beta S))\to H(\beta S)$
and $\operatorname{pr}_2\colon
\operatorname{Gr}_{\operatorname{inv}}(H(\beta S))\to H(\beta S)$
are homeomorphisms. Consequently, the inversion
$\operatorname{inv}|_{H(\beta S)}=\operatorname{pr}_2\circ
(\operatorname{pr}_1)^{-1}\colon H(\beta S)\to H(\beta S)$ is
continuous, being a composition of homeomorphisms. Therefore the
inversion $\operatorname{inv}\colon H(S)\rightarrow H(S)$ is
continuous as a restriction of a continuous map.

\smallskip
$(ii)$ The projection $\pi\colon H(S)\rightarrow E(S)$ is
continuous as a composition of two continuous maps.

\smallskip
$(iii)$ Given an idempotent $e\in E(S)$, consider the maximal
subgroup $H_\beta(e)$  in $\beta S$ containing $e$. Then by
Theorems~1.11 and 1.13 from
\cite[Vol.~1]{CarruthHildebrantKoch1983-1986}, $H_\beta(e)$ is a
compact topological group and since $H(e)$ is a subgroup of
$H_\beta(e)$ the inversion $\operatorname{inv}\colon
H(e)\rightarrow H(e)$ is continuous, and $H(e)$ is a totally
bounded topological group, see \cite{Weil1937}.
\end{proof}

Theorem~\ref{th1} implies the following:

\begin{corollary}\label{cor1-1}
If $S$ is a Tychonoff Clifford topological semigroup with the
pseudocompact square $S\times S$ then the inversion in $S$ is
continuous.
\end{corollary}

An element $x$ of a topological semigroup $S$ is called
\emph{topologically periodic} if for any open neighbourhood $U(x)$
of $x$ there exists an integer $n\geqslant 2$ such that $x^n\in
U(x)$. A topological semigroup $S$ is called \emph{topologically
periodic} if any element of $S$ is topologically periodic.

\begin{remark}\label{rem1-2}
The following observation implies that \emph{an element of any
(not necessarily Hausdorff) topological semigroup $S$ is
topologically periodic if and only if for any integer $n\geqslant
2$ and for any open neighbourhood $U(x)$ of $x$ there exists an
integer $m\geqslant n$ such that $x^m\in U(x)$.} 
Let $k\geqslant
2$ be an integer
such that $x^k\in U(x)$. Then the continuity of the
semigroup operation implies that there exists an open
neighbourhood $V(x)$ of $x$ such that $\big(V(x)\big)^k\subseteq
U(x)$. Since $x$ is topologically periodic there exists an integer
$m\geqslant 2$ such that $x^m\in V(x)$. Hence we have
$x^{km}\in\big(V(x)\big)^k\subseteq U(x)$ and $km\geqslant
4=2^{2}$. Proceeding by induction, we can find an integer
$p\geqslant 2^n>n$ such that $x^p\in U(x)$.
\end{remark}

\begin{theorem}\label{th2}
Let $S$ be a Hausdorff topological semigroup with the countably
compact square $S\times S$. Then:
\begin{itemize}
    \item[$(i)$] each maximal subgroup $H(e)$ of $S$ is a
                 countably compact topological group; and
    \item[$(ii)$] the subset $H(S)$ is countably compact.
\end{itemize}
\end{theorem}

\begin{proof}
$(i)$ Let $H(e)$ be any maximal subgroup of $S$. Since the
semigroup operation in $S$ is continuous the subset
\begin{equation*}
    G=\{(x,y)\in S\times S\mid xy=yx=e,\; xe=ex=x,\; ye=ey=y\}
\end{equation*}
is closed in $S\times S$ and Theorem~3.10.4 from~\cite{Engelking1989}
implies that $G$ is a countably compact subset in $S\times S$.
Consider the natural projection $\operatorname{pr}_1\colon S\times
S\rightarrow S$ onto the first coordinate. Since
$\operatorname{pr}_1(G)=H(e)$ and the projection
$\operatorname{pr}_1\colon S\times S\rightarrow S$ is a continuous
map Theorem~3.10.5 from~\cite{Engelking1989} implies that $H(e)$ is a
countably compact subspace of $S$.

Next, we show that $H(e)$ is a topologically periodic
paratopological group. Let $x$ be an arbitrary element of the
subgroup $H(e)$ and $U(x)$ be any open neighbourhood of $x$. We
consider the sequence $\{(x^{n+1},x^{-n})\}_{n=1}^\infty$ in
$H(e)\times H(e)\subseteq S\times S$. The countable compactness of
$S\times S$ guarantees that this sequence has an accumulation
point $(a, b)\in S\times S$. Since $x^{n+1}\cdot x^{-n}=x$,
the
continuity of the semigroup operation on $S$ guarantees that $ab=
x$. Then for any open neighbourhood $U(x)$ of $x$ in $S$ there
exist open neighbourhoods $U(a)$ and $U(b)$ of the point $a$ and
$b$ in $S$, respectively, such that $U(a)U(b)\subseteq U(x)$.
Since $(a, b)$ is an accumulation point of the sequence
$\{(x^{n+1},x^{-n})\}_{n=1}^\infty$ in $S\times S$,
there exist positive integers $m$ and $n$ such that $x^m\in U(a)$,
$x^{-n}\in U(b)$ and $m\geqslant n+2$. Hence we get that $x^m\cdot
x^{-n}=x^{m-n}\in U(a)\cdot U(b)\subseteq U(x)$ and $m-n\geqslant
2$. Therefore $H(e)$ is a topologically periodic paratopological
group. By Bokalo-Guran Theorem (see
\cite[Theorem~3]{BokaloGuran1996}) any countably compact
paratopological group is a topological group. Consequently, $H(e)$
is a countably compact topological group.

\smallskip
$(ii)$ Since the semigroup operation in $S$ is continuous,
\begin{equation*}
    H=\{(x,y)\in S\times S\mid xyx=x,\; yxy=y,\; xy=yx\in E(S)\}
\end{equation*}
is a closed subset in $S\times S$ and
Theorem~3.10.4 from~\cite{Engelking1989} implies that $H$ is a
countably compact subset in $S\times S$. Consider the natural
projection $\operatorname{pr}_1\colon S\times S\rightarrow S$ onto
the first coordinate. Since $\operatorname{pr}_1(H)=H(S)$ and the
projection $\operatorname{pr}_1\colon S\times S\rightarrow S$ is a
continuous map, Theorem~3.10.5 from~\cite{Engelking1989} implies that
$H(S)$ is a countably compact subspace of $S$.
\end{proof}

\begin{proposition}\label{pr2-1}
Let $x$ be a topologically periodic element of a maximal subgroup
$H(e)$ with the unity $e$ in a topological semigroup $S$. Then the
inversion $\operatorname{inv}\colon H(S)\rightarrow H(S)$ is
continuous at $x$ if and only if it is continuous at the
idempotent $e$.
\end{proposition}

\begin{proof}
We follow the argument of \cite{BanakhGutik2004}. Let $U(x^{-1})$
be any open neighbourhood of the inverse element $x^{-1}$ of $x$
in $S$. Since the semigroup operation in $S$ is continuous there
exist open neighbourhoods $V(x^{-1})$ and $V(e)$ of $x^{-1}$ and
$e$ in $H(S)$, respectively, such that $V(x^{-1})\cdot
V(e)\subseteq U(x^{-1})$. Since the inversion is continuous at
idempotent $e$ there exists an open neighbourhood $W(e)$ of $e$ in
$H(S)$ such that $(W(e))^{-1}\subseteq V(e)$. 

Also, the continuity
of the semigroup operation implies that there exists an open
neighbourhood $N(x)$ of $x$ in $H(S)$ such that $x^{-1}\cdot
N(x)\cdot x^{-1}\subseteq V(x^{-1})$ and $N(x)\cdot
x^{-1}\subseteq W(e)$. The topological periodicity of $x$ implies
that there exists a positive integer $n$ such that $x^{n+2}\in
N(x)$. Then we have that
\begin{equation*}
    x^{n+1}=x^{n+2}\cdot x^{-1}\in N(x)\cdot x^{-1}\subseteq W(e)
\end{equation*}
and
\begin{equation*}
    x^n=x^{-1}\cdot x^{n+2}\cdot x^{-1}\in x^{-1}\cdot N(x)\cdot
    x^{-1}\subseteq V(x^{-1}).
\end{equation*}
Since $S$ is a topological semigroup there exists an open
neighbourhood $P(x)$ of $x$ in $S$ such that $\big(H(S)\cap
P(x)\big)^{n+1}\subseteq W(e)$ and $\big(H(S)\cap
P(x)\big)^{n}\subseteq V(x^{-1})$. Therefore we get
\begin{equation*}
\begin{split}
    \big(H(S)\cap P(x)\big)^{-1}&\subseteq \big(H(S)\cap
    P(x)\big)^{n}\cdot
    \Big(\big(H(S)\cap P(x)\big)^{n+1}\Big)^{-1}\subseteq V(x^{-1})\cdot
    \big(W(e)\big)^{-1}\subseteq\\
    &\subseteq V(x^{-1})\cdot V(e)\subseteq
    U(x^{-1}),
\end{split}
\end{equation*}
and hence the inversion is continuous at the point $x$.
\end{proof}

Proposition~\ref{pr2-1} implies the following:

\begin{corollary}\label{cor2-2}
The inversion in a topologically periodic Clifford topological
semigroup $S$ is continuous if and only if it is continuous at any
idempotent of the semigroup $S$.
\end{corollary}

\begin{theorem}\label{th3}
Let $S$ be a regular topological semigroup with the countably
compact square $S\times S$. Then:
\begin{itemize}
    \item[$(i)$] the inversion $\operatorname{inv}\colon
                 H(S)\rightarrow H(S)$ is continuous; and
    \item[$(ii)$] the projection $\pi\colon
                 H(S)\rightarrow E(S)$ is continuous.
\end{itemize}
\end{theorem}

\begin{proof}
$(i)$ By Proposition~\ref{pr2-1} it is sufficient to show that the
inversion $\operatorname{inv}\colon H(S)\rightarrow H(S)$ is
continuous at any point of the set $E(S)$.

Fix any $e\in E(S)$. Let $U(e)$ be any open neighbourhood of $e$
in $S$. Since the topological space of the semigroup $S$ is
regular, the continuity of the semigroup operation of $S$ implies
that there exists a sequence of open neighbourhoods
$\{U_i(e)\}_{i=1}^\infty$ of the idempotent $e$ in $S$ such that
$\overline{U_1(e)}\subseteq U(e)$ and
$\overline{\big(U_n(e)\big)^m}\subseteq U_{n-1}(e)$ for any
positive integer $n$ and all $m=1,\ldots,n$. 
Let
$F=\displaystyle\bigcap_{n=1}^\infty\overline{U_n(e)}$. 

We shall
show that $\big(F\cap H(S)\big)^{-1}\subseteq F$. Let $x$ be any
element of the set $F\cap H(S)$. Since the set $F$ is closed to
prove that $x^{-1}\in F$ it sufficient to show that $V(x^{-1})\cap
F\neq\varnothing$ for any open neighbourhood $V(x^{-1})$ of the
point $x^{-1}$. The continuity of the semigroup operation in $S$
and the equality $x^{-1}=x^{-1}\cdot x\cdot x^{-1}$ imply that
there exists an open neighbourhood $V(x)$ of the point $x$ in $S$
such that $x^{-1}\cdot V(x)\cdot x^{-1}\subseteq V(x^{-1})$. By
Theorem~\ref{th2}~$(i)$ the element $x$ of $S$ is topologically
periodic, and hence there exists a positive integer $n\geqslant 2$
such that $x^n\in V(x)$. Then we have
\begin{equation*}
    x^{n-2}=x^{-1}\cdot x^n\cdot x^{-1}\in x^{-1}\cdot V(x)\cdot
    x^{-1}\subseteq V(x^{-1})
\end{equation*}
and
\begin{equation*}
    x^{n-2}\in F^{n-2}\subseteq\bigcap_{i=n-2}^\infty
    \big(U_i(e)\big)^{n-2}\subseteq \bigcap_{i=n-2}^\infty
    U_{i-1}(e)\subseteq F.
\end{equation*}
Hence $V(x^{-1})\cap F\neq\varnothing$ and since  $F$ is a closed
subset in $S$ we have that $x^{-1}\in F$. This implies that the
inclusion $\big(F\cap H(S)\big)^{-1}\subseteq F$ holds.

Later we shall show that $\big(U_n(e)\cap H(S)\big)^{-1}\subseteq
U(e)\cap H(S)$ for some positive integer $n$. Suppose to the
contrary that $\big(U_n(e)\cap H(S)\big)^{-1}\nsubseteq U(e)\cap
H(S)$ for any positive integer $n$. Then there exists a sequence
$\{x_n\}_{n=1}^\infty$ in $H(S)$ such that $x_n\in
U_n(e)\setminus\big(U(e)\big)^{-1}$ for all positive integers $n$.

The countable compactness of the square $S\times S$ implies that
the sequence $\{(x_n,x_n^{-1})\}_{n=1}^\infty$ has a cluster point
$(a,b)$ in $S\times S$. The continuity of the semigroup operation
in $S$ implies that
\begin{equation*}
  a\cdot b=b\cdot a=f,\;  a\cdot b\cdot a=a,\; b\cdot a\cdot b=b, \;
\end{equation*}
and hence $a, b\in H(f)$ for some idempotent $f$ in $S$. Therefore
$b=a^{-1}\in F^{-1}\cap H(S)\subseteq F$. Then $(a,b)\in F\times
F\subseteq U(e)\times U(e)$. Since $(a,b)$ is a cluster point of
the sequence $\{(x_n,x_n^{-1})\}_{n=1}^\infty$, there exists a
positive integer $n$ such that $(x_n,x_n^{-1})\in U(e)\times
U(e)$. 

Therefore we have that $x_n\in \big(U(e)\big)^{-1}$ which
contradicts the choice of the sequence $\{x_n\}_{n=1}^\infty$. The
obtained contradiction implies that $\big(U_n(e)\cap
H(S)\big)^{-1}\subseteq U(e)\cap H(S)$ for some positive integer
$n$, and hence the inversion $\operatorname{inv}\colon
H(S)\rightarrow H(S)$ is continuous.

\smallskip
$(ii)$ The projection $\pi\colon H(S)\rightarrow E(S)$ is
continuous as a composition of two continuous maps.
\end{proof}

Theorem~\ref{th3} implies the following corollary generalizing a
result of~\cite{BanakhGutik2004}.

\begin{corollary}\label{cor3-1}
The inversion in a regular Clifford topological semigroup with the
countably compact square is continuous.
\end{corollary}

Let $S$ be a topological semigroup and $e\in E(S)$. We shall say
that the semigroup $S$ is \emph{inversely regular at} $e$ if for
any open neighbourhood $U(e)$ of $e$ there exists an open
neighbourhood $W(e)$ of $e$ such that $\overline{\big(W(e)\cap
H(S)\big)^{-1}}\subseteq \big(U(e)\cap H(S)\big)^{-1}$. A
topological semigroup $S$ with non-empty subsets of idempotents is
called \emph{inversely regular} if it is inversely regular at each
idempotent of $S$~\cite{BanakhGutik2004}.

\begin{theorem}\label{th4}
Let $S$ be a topologically periodic Hausdorff topological
semigroup. If $S$ is inversely regular and countably compact,
then:
\begin{itemize}
    \item[$(i)$] the inversion $\operatorname{inv}\colon
                 H(S)\rightarrow H(S)$ is continuous;
    \item[$(ii)$] the projection $\pi\colon
                 H(S)\rightarrow E(S)$ is continuous.
\end{itemize}
\end{theorem}

\begin{proof}
$(i)$ Fix any idempotent $e$ in $S$. Let $U(e)$ be any open
neighbourhood of $e$ in $S$. Since the semigroup operation in $S$
is continuous and $S$ is inversely regular we construct
inductively two sequences $\{U_n(e)\}_{i=1}^\infty$ and
$\{W_n(e)\}_{i=1}^\infty$ of open neighbourhoods of the idempotent
$e$ such that $\big(U_n(e)\big)^i\subseteq W_{n-1}(e)$ and
$\overline{\big(W_n(e)\cap H(S)\big)^{-1}}\subseteq
\big(U_n(e)\cap H(S)\big)^{-1}$ for all positive integers $n$ and
$i=1,2,\ldots,n$.

Let $F=\displaystyle\bigcap_{n=1}^\infty\overline{\big(W_n(e)\cap
H(S)\big)^{-1}}$. Then we have that
\begin{equation*}
    F=\bigcap_{n=1}^\infty\overline{\big(W_n(e)\cap
    H(S)\big)^{-1}}\subseteq \bigcap_{n=2}^\infty \big(U_n(e)\cap
    H(S)\big)^{-1}\subseteq \bigcap_{n=2}^\infty
    \big(W_{n-1}(e)\cap H(S)\big)^{-1}\subseteq F.
\end{equation*}
We shall show that $F^{-1}=F$. Let $x$ be arbitrary element of
$F$. Since the set $F$ is closed it sufficient to prove that
$V(x^{-1})\cap F\neq\varnothing$ for any open neighbourhood
$V(x^{-1})$ of the point $x^{-1}$. Since the semigroup $S$ is
topologically periodic there exists a positive integer $n\geqslant
2$ such that $x^{n-2}\in V(x^{-1})$ (see proof of
Theorem~\ref{th3}). Then we have
\begin{equation*}
\begin{split}
    x^{n-2}\in F^{n-2}&
    \subseteq \bigcap_{k=n-2}^\infty \Big(\big(U_k(e)\cap
    H(S)\big)^{-1}\Big)^{n-2}
    \subseteq \bigcap_{k=n-2}^\infty \Big(\big(U_k(e)\cap
    H(S)\big)^{n-2}\Big)^{-1}\subseteq\\
    &\subseteq \bigcap_{k=n-2}^\infty
    \big(W_{k-1}(e)\cap H(S)\big)^{-1}\subseteq F.
\end{split}
\end{equation*}
Hence $V(x^{-1})\cap F\neq\varnothing$ and since  $F$ is a closed
subset in $S$ we have that $x^{-1}\in F$. This implies that the
inclusion $F^{-1}\subseteq F$ holds. Then after the inversion we
get that $F\subseteq F^{-1}$. Therefore we get that
\begin{equation*}
    \bigcap_{n=1}^\infty\big(W_{n}(e)\cap H(S)\big)^{-1}=F=F^{-1}=
    \bigcap_{n=1}^\infty\Big(\overline{\big(W_n(e)\cap
    H(S)\big)^{-1}}\Big)^{-1}\subseteq \bigcap_{n=1}^\infty \big(U_n(e)\cap
    H(S)\big)\subseteq U(e).
\end{equation*}
Since the space of the semigroup $S$ is countably compact there
exists a positive integer $n$ such that
\begin{equation*}
 F\subseteq\big(W_n(e)\cap H(S)\big)^{-1}\subseteq
 \overline{\big(W_n(e)\cap H(S)\big)^{-1}}\subseteq U(e).
\end{equation*}
This implies that the inversion is continuous at the idempotent
$e$.

\smallskip
$(ii)$ The projection $\pi\colon H(S)\rightarrow E(S)$ is
continuous as a composition of two continuous maps.
\end{proof}

We recall that a map $f\colon X\rightarrow Y$ between topological
space is called \emph{sequentially continuous} if
$\displaystyle\lim_{n\to\infty} f(x_n)=f\big(\lim_{n\to\infty}
x_n\big)$ for any convergent sequence $\{ x_n\}_{n=1}^\infty$ in
$X$. Obviously,
a composition of two sequentially continuous
maps is a continuous map. A subset $F$ of a topological space $X$
is called \emph{sequentially closed} if no sequence in $F$
converges to a point not in $F$~\cite{Franklin1965}.

\begin{theorem}\label{th5}
Let $S$ be a Hausdorff countably compact topological semigroup.
Then the following conditions hold:
\begin{itemize}
    \item[$(i)$] each maximal subgroup $H(e)$, $e\in E(S)$, is
                 sequentially closed in $S$;
    \item[$(ii)$] the subset $H(S)$ is sequentially closed in
                 $S$;
    \item[$(iii)$] the inversion $\operatorname{inv}\colon
                 H(S)\rightarrow H(S)$ is sequentially
                 continuous; and
    \item[$(iv)$] the projection $\pi\colon
                 H(S)\rightarrow E(S)$ is sequentially
                 continuous.
\end{itemize}
\end{theorem}

\begin{proof}
$(i)$ Suppose to the contrary that there exists a maximal
subgroup $H(e)$ in $S$ which is not sequentially closed subset in
$S$. Then there exists a sequence $\{ x_n\}_{n=1}^\infty\subset
H(e)$ which converges to $x\notin H(e)$. We put $A=\{
x_n\}_{n=1}^\infty\cup\{ x\}$. Since the sequence $\{
x_n\}_{n=1}^\infty\subset H(e)$ converges to $x$ the set $A$ with
the induced from $S$ topology is a compact space. Then by
Corollary~3.10.14 from~\cite{Engelking1989}, $A\times S$ is a countably
compact space. 

The countable compactness of $A\times S$ implies
that the sequence $\{(x_n,x_n^{-1})\}_{n=1}^\infty$ has a cluster
point $(a,b)$ in $A\times S$. The continuity of the semigroup
operation in $S$ implies that $ab=e$, $aba=a$, $bab=b$ and hence
$a=b^{-1}$ and $a,b\in H(e)$. The Hausdorff property of $S$
implies that $x=a$ and hence $x\in H(e)$. The obtained
contradiction implies assertion $(i)$.

\smallskip
$(ii)$ We argue exactly as a previous case. Suppose to the
contrary that $H(S)$ is not a sequentially closed subset in $S$.
Then there exists a sequence $\{ x_n\}_{n=1}^\infty\subset H(S)$
which converges to $x\notin H(S)$. We put $A=\{
x_n\}_{n=1}^\infty\cup\{ x\}$. Since the sequence $\{
x_n\}_{n=1}^\infty\subset H(S)$ converges to $x$ the set $A$ with
the induced from $S$ topology is a compact space. Then by
Corollary~3.10.14 from~\cite{Engelking1989}, $A\times S$ is a countably
compact space. 

The countable compactness of $A\times S$ implies
that the sequence $\{(x_n,x_n^{-1})\}_{n=1}^\infty$ has a cluster
point $(a,b)$ in $A\times S$. The continuity of the semigroup
operation in $S$ implies that $ab=e$, $aba=a$, $bab=b$ for some
idempotent $e\in E(S)$ and hence $a=b^{-1}$ and $a,b\in H(e)$. The
Hausdorff property of $S$ implies that $x=a$ and hence $x\in
H(e)\subseteq H(S)$. The obtained contradiction implies assertion
$(ii)$.

\smallskip
$(iii)$ The sequential continuity of the inversion
$\operatorname{inv}\colon H(S)\to H(S)$ will follow as soon as we
prove that for any countable compactum $C\subset H(S)$ the
restriction $\operatorname{inv}|_C$ is continuous. Let
\begin{equation*}
\operatorname{Gr}_{\operatorname{inv}}(S)=\{(x, y)\in S\times
S\mid y=x^{-1}\}
\end{equation*}
be the graph of the inversion. Since $S$ is a topological
semigroup and
\begin{equation*}
\operatorname{Gr}_{\operatorname{inv}}(S)=\{(x, y)\in S\times
S\mid xyx=x, yxy=y\mbox{ and } xy=yx\},
\end{equation*}
the graph $\operatorname{Gr}_{\operatorname{inv}}(S)$ is a closed
subset of $S\times S$.

Since $C$ is a metrizable compactum we can apply
Corollary~3.10.14 from~\cite{Engelking1989} to conclude that $C\times
S$ is a countably compact space. Then the closedness of
$\operatorname{Gr}_{\operatorname{inv}}(S)$ in the space $S\times
S$ implies that the space $G=(C\times S)\cap
\operatorname{Gr}_{\operatorname{inv}}(S)$ is countably compact
and being countable, is compact.

Consider the natural projections $\operatorname{pr}_1\colon
S\times S\to S$ and $\operatorname{pr}_2\colon S\times S\to S$
onto the first and the second coordinates, respectively. It
follows from the compactness of $G$ that
$\operatorname{pr}_1\colon G\to C$ and $\operatorname{pr}_2\colon
G\to C^{-1}$ are homeomorphisms. Consequently,
$\operatorname{inv}|_C=\operatorname{pr}_2\circ
(\operatorname{pr}_1)^{-1}\colon C\to C^{-1}$ is continuous, being
a composition of homeomorphisms.

\smallskip
$(iv)$ The projection $\pi\colon H(S)\rightarrow E(S)$ is
sequentially continuous as a composition of two sequentially
continuous maps.
\end{proof}

Theorem~\ref{th5} implies the following:

\begin{corollary}\label{cor5-1}
Let $S$ be a Hausdorff Clifford countably compact topological
semigroup. Then the following conditions hold:
\begin{itemize}
    \item[$(i)$] each maximal subgroup $H(e)$ is sequentially
                 closed in $S$;
    \item[$(ii)$] the inversion $\operatorname{inv}\colon
                 S\rightarrow S$ is sequentially continuous; and
    \item[$(iii)$] the projection $\pi\colon
                 S\rightarrow E(S)$ is sequentially continuous.
\end{itemize}
\end{corollary}

We observe that any sequentially compact (and any hence sequential
countably compact) topological semigroup contains an idempotent
(see \cite[Theorem~8]{BanakhDimitrovaGutik2009x}). For sequential
countably compact semigroup Theorem~\ref{th5} implies the
following:

\begin{corollary}\label{cor5-2}
Let $S$ be a Hausdorff sequential countably compact topological
semigroup. Then the following conditions hold:
\begin{itemize}
    \item[$(i)$] each maximal subgroup $H(e)$ is closed in $S$;
    \item[$(ii)$] the subset $H(S)$ is closed in $S$;
    \item[$(iii)$] the inversion $\operatorname{inv}\colon
                 H(S)\rightarrow H(S)$ continuous; and
    \item[$(iv)$] the projection $\pi\colon
                 H(S)\rightarrow E(S)$ is continuous.
\end{itemize}
\end{corollary}

The following example shows that the closure of a subgroup of a
countably compact topological semigroup need not be a subgroup.

\begin{example}\label{example11Tomita}
Assume $\textsf{MA}\,_{\textrm{countable}}$ holds. Let
$(\mathbb{R},+)$ be the additive topological group of real number
with usual topology and $\mathbb{Z}$ the discrete additive
group of integer. Then $\mathbb{T}=\mathbb{R}/\mathbb{Z}$ is a
topological group. Let $G$ be the group consisting of all
$y\in\mathbb{T}^{\mathfrak{c}}$ such that there exists
$\mu\in\mathfrak{c}$ such that $y(\mu)$ is the identity of
$\mathbb{T}$ for each $\alpha>\mu$.

There exists $x\in\mathbb{T}^{\mathfrak{c}}$ such that  $S=\{
nx+y\mid n\in\omega \hbox{ and } y\in G\}$ is the semigroup with
two-sided cancelation but $S$ is not a group, see
\cite{RobbieSvetlichny1996, Tomita1996}. Since $G$ is dense
subgroup in $\mathbb{T}^{\mathfrak{c}}$, $G$ is dense in $S$.
Tomita 
\cite{Tomita1996} showed that the existence of an
element $x$ in $S$ is independent of $\textsf{ZFC}$. Also
Madariaga-Garcia and Tomita 
\cite{Madariaga-Garcia_Tomita2007}
show that the semigroup $S$ can be constructed from $\mathfrak{c}$
selective ultrafilters.
\end{example}

A topological space $X$ is called \emph{quasi-regular} if for any
non-empty open subset $U$ in $X$ there exists a non-empty open set
$V\subseteq U$ in $X$ such that $\overline{V}\subseteq U$. The
following example shows that there exists a Hausdorff
quasi-regular Clifford inverse countably compact topological
semigroup $S$ with the discontinuous inversion and discontinuous
projection $\pi\colon S\rightarrow E(S)$.

\begin{example}[\cite{BanakhGutik2004}]\label{ex1} Let $\omega_1$
be the smallest uncountable ordinal and $[0,\omega_1)$ be the
well-ordered sets of all countable ordinals, endowed with the
natural order topology (see
\cite[Example~3.10.16]{Engelking1989}). It is well-known that
$[0,\omega_1)$ is a sequentially compact topological space (see:
\cite{Engelking1989}, Example~3.10.16) and simple verification
shows that the semilattice operation $\min$ is continuous on
$[0,\omega_1)$.

Let $\mathbb{T}=\{z\in \mathbb{C} : |z|=1\}$ be the unit circle in
the complex plane with the usual topology and let $\mathbb{T}$
endowed with the operation of multiplication of complex numbers.
Then by Theorem~3.10.35 from~\cite{Engelking1989} the product
$A=[0,\omega_1)\times \mathbb{T}$ is an Hausdorff sequentially
compact commutative topological inverse semigroup, as the
Cartesian product of a topological semilattice and a commutative
topological group.

Let $x=(\omega_1,1)\not\in A$. Put $S=A\cup\{ x\}$ and define a
topology $\tau$ on $S$ letting $A$ be a subspace of $S$ and
$U\subset S$ be a neighborhood of $x$ if there are a positive real
$\varepsilon
>0$ and a countable ordinal $\alpha$ such that $U\supset
U(\alpha,\varepsilon)$ where
\begin{equation*}
U(\alpha,\varepsilon)= \{x\}\cup\{(\beta,e^{i\varphi})\mid
\alpha<\beta<\omega_1,\;\;0<\varphi<\varepsilon\}.
\end{equation*}

Extend the semigroup operation to $S$ letting $x\cdot x=x$ and
$x\cdot a=a\cdot x=a$ for all $a\in A$. It is easy to see that $S$
is a sequentially compact space and the semigroup operation
"$\cdot$" is continuous and commutative on $S$. But the inversion
in $S$ is not continuous since $\left( U(\alpha, \varepsilon
)\right)^{-1}\not\subseteq U(\beta, \delta)$ for all $\alpha,
\beta<\omega_1$ and $\varepsilon, \delta\in(0, 1)$.

Observe also that the subsemigroup of idempotents of the semigroup
$S$ can be identified with the discrete sum of
$[0,\omega_1)\bigsqcup\{\omega_1\}$ and hence is sequentially
compact, locally countable, and locally compact.

Also we observe that the projective map $\pi\colon S\rightarrow
E(S)$ is not continuous.
\end{example}

\begin{remark} Example~\ref{ex1} shows that the requirement of regularity
in Theorems 2 and Corollary 2 is essential and cannot be replaced
by the quasi-regularity. This contrasts with the case of
paratopological groups - see \cite{RavskyReznichenko2002}.
\end{remark}


\section*{Acknowledgements}

This research was supported by the Slovenian Research Agency grants
P1-0292-0101-04, 
J1-9643-0101 and BI-UA/07-08/001.
The authors gratefuly acknowledge the referee for several
comments and suggestions which have considerable improve the original
version of the manuscript.


\end{document}